\documentclass[10pt, article]{amsart}
\usepackage{tikz}
\usetikzlibrary{calc}
\usepackage{ae} 
\usepackage[T1]{fontenc}
\usepackage[cp1250]{inputenc}
\usepackage{amsmath}
\usepackage{amssymb, amsfonts,amscd,verbatim}

\usepackage[normalem]{ulem}
\usepackage{hyperref}
\usepackage{indentfirst}
\usepackage{latexsym}
\input xy
\xyoption{all}

\usepackage{amsmath}    

\theoremstyle{plain}
\newtheorem{Pocz}{Poczatek}[section]
\newtheorem{Proposition}[Pocz]{Proposition}

\newtheorem{Theorem}[Pocz]{Theorem}
\newtheorem{Corollary}[Pocz]{Corollary}

\newtheorem{Lemma}[Pocz]{Lemma}
\newtheorem{Observation}[Pocz]{Observation}

\newtheorem{Example}[Pocz]{Example}

\theoremstyle{definition}
\newtheorem{Definition}[Pocz]{Definition}

\theoremstyle{remark}
\newtheorem{Remark}[Pocz]{Remark}

\DeclareMathOperator*{\diam}{diam}

\numberwithin{equation}{section}
\title[Coarse Freundenthal compactification and ends of groups]
{Coarse Freundenthal compactification and ends of groups}

\author{Yuankui Ma}
\address{Xi'an Technological University, No.2 Xuefu zhong lu, Weiyang district, Xi'an, China 710021}
\email{mayuankui@xatu.edu.cn}

\author{Jerzy Dydak}
\address{University of Tennessee, Knoxville, TN 37996, USA}
\email{jdydak@utk.edu}
\address{Xi'an Technological University, No.2 Xuefu zhong lu, Weiyang district, Xi'an, China 710021}
\email{jdydak@gmail.com}

\date{ \today
}
\keywords{dimension, coarse geometry, ends of groups, Freundenthal compactification, Higson corona}

\subjclass[2000]{Primary 54D35; Secondary 20F69}

\begin{document}
\maketitle
\begin{center}
\today
\end{center}

\tableofcontents

\begin{abstract}
A coarse compactification of a proper metric space $X$ is any compactification of $X$ that is dominated by its Higson compactification. In this paper we describe the maximal coarse compactification of $X$ whose corona is of dimension $0$. In case of geodesic spaces $X$, it coincides with the Freundenthal compactification of $X$. As an application we provide an alternative way of extending the concept of the number of ends from finitely generated groups to arbitrary countable groups. We present a geometric proof of a generalization of Stallings' theorem by showing that any countable group of two ends contains an infinite cyclic subgroup of finite index.
Finally, we define ends of arbitrary coarse spaces.
\end{abstract}

\section{Introduction}

Historically, as noted in \cite{DK} on p.287, 
ends are the oldest coarse topological notion and were used by Freundenthal in 1930 in his famous compactification (see \cite{Peschke}  for information about theorems in this section):

\begin{Theorem}
Suppose $X$ is a $\sigma$-compact locally compact and locally connected Hausdorff space. It has a compactification $\bar X$ such that
$\bar X\setminus X$ is of dimension $0$ and $\bar X$ dominates any compactification
$\hat X$ of $X$ whose corona is of dimension $0$.
\end{Theorem}

\begin{Definition}
A \textbf{Freundenthal end} is a decreasing sequence $\{U_i\}_{i\ge i}$ of components of sets $X\setminus K_i$, where $K_i$ are compact, $K_i\subset int(K_{i+1})$ for each $i\ge 1$, and $$\bigcup\limits_{i=1}^\infty K_i=X.$$
The space of ends of $X$ is denoted by $Ends(X)$
\end{Definition}

The topology on $X\cup Ends(X)$ is induced by the following basis:\\
1. It includes all open subsets of $X$ with compact closure,\\
2. It includes any component $U$ of $X\setminus K_i$ union all ends containing $U$.

\begin{Theorem}
(Freudenthal) A path connected topological group has at most two ends.
\end{Theorem}

\begin{Theorem}
 (H. Hopf) Let $G$ be a finitely generated discrete group acting on a space $X$ by covering transformations. Suppose the orbit space $B:=X/G$ is compact. Then (i) and (ii), below, hold.\\
(i) The end space of $X$ has 0, 1 or 2 (discrete) elements or is a Cantor space.\\ 
(ii) If $G$ also acts on $Y$ satisfying the hypotheses above, then $X$ and $Y$ have
homeomorphic end spaces.
\end{Theorem}

Conclusion (ii) suggests to regard the end space of $X$ as an invariant of the group $G$ itself:
\begin{Definition}
Let $p:X \to B$ be a covering map with compact base $B$ and the group of covering transformations $G$. The \textbf{end space} of $G$ is
$$Ends(G):= Ends(X).$$
\end{Definition}

When applied to a Cayley graph of $G$, it gives the standard definition of ends of finitely generated groups (see \cite{DK}, p.295). See \cite{Geog}  for basic results in this theory and see \cite{Grom} for more general facts in coarse geometry related to groups.

 In this paper we will define ends of arbitrary countable groups by generalizing the construction of the Higson corona. In the case of coarse spaces we generalize Freudenthal's method to define their space of ends.

E. Specker \cite{Sp} defined ends of arbitrary groups using Stone's duality theorem. See a very nice paper \cite{Corn} of Yves Cornulier describing properties of the space of ends of infinitely generated groups. We consider Specker's approach highly non-geometric. Additionally, our way of defining ends of spaces leads directly to view them as coronas of certain compactifications (large scale compactifications in case of coarse spaces). A future paper will demonstrate the equivalence of Specker's definition of ends of groups and our definition of them.

The authors are grateful to Ross Geoghegan and Mike Mihalik for their help in understanding classical theory of ends of finitely generated groups.

\section{Coarse compactifications}
In this section we define the concept of a coarse compactification of a proper metric space $X$ and we give necessary and sufficient condition for the Freundenthal compactification to be a coarse one.
\begin{Definition}
A \textbf{coarse compactification} of a proper metric space $X$ is any compactification $\bar X$ of $X$ that is dominated by its Higson compactification.
Equivalently, any continuous function $f:\bar X\to R$ restricts
to a slowly oscillating function $f|X:X\to R$.
\end{Definition}

Recall $g:X\to R$ is \textbf{slowly oscillating} if for any $r, \epsilon > 0$ there is a bounded subset $K$ of $X$ such that $x,y\in X\setminus K$ and $d(x,y) < r$ implies $|g(x)-g(y)| < \epsilon$.

The \textbf{Higson compactification} $h(X)$ of $X$ is the one induced by all continuous and slowly oscillating functions $f:X\to [0,1]$. Equivalently,
all continuous functions $g:h(X)\to R$ are slowly oscillating when restricted to $X$ and every continuous and slowly oscillating function $f:X\to [0,1]$ extends over $h(X)$ to a continuous function.

Let us show a necessary and sufficient condition for the Freundenthal compactification to be a coarse compactification.

\begin{Theorem}\label{CoarseFreundenthal}
Suppose $\bar X$ is the Freundenthal compactification of a proper, connected, and locally connected metric space $X$. The following conditions are equivalent:\\
1. $\bar X$ is a coarse compactification of $X$,\\
2. For each $m > 0$ and each bounded subset $K$ of $X$ there is a bounded subset $L\supset K$ of $X$ such that for every $x\in X\setminus L$ the $m$-ball $B(x,m)$ is contained in a component of $X\setminus K$.
\end{Theorem}
\begin{proof}
2)$\implies$1).
Suppose $g:\bar X\to R$ is continuous and $\epsilon > 0$. Since $\dim(\bar X\setminus X)=0$ there are mutually disjoint open sets $U_i$, $i=1,\ldots,n$,
such that $\bar X\setminus X\subset \bigcup\limits_{i=1}^n U_i$
and the diameter of each $g(U_i)$ is less than $\epsilon$.
Notice $K:=X\setminus \bigcup\limits_{i=1}^n U_i$ is a compact subset of $X$, so given $m > 0$ there is $L\supset K$ such that $B(x,m)$ is contained in a component of $X\setminus K$ if $x\in X\setminus L$ and that component is contained in some $U_i$. Therefore $|g(x)-g(y)| < \epsilon$ if $y\in B(x,m)$ and $g$ is slowly oscillating.\\
1)$\implies$2).
If $\bar X$ is a coarse compactification of $X$, $K$ is a bounded subset of $X$, and $m > 0$, then assume existence of two sequences $x_n, y_n$
such that $d(x_n,y_n) < m$ for each $n\ge 1$, $x_n\in X\setminus K_i$ is contained 
in a component $C_i$ of $X\setminus cl(K)$, $y_n$ is contained 
in a component $D_i$ of $X\setminus cl(K)$, where $C_i\ne D_i$.
Since there are only finitely many unbounded components of $X\setminus cl(K)$,
we may assume $C_i=C$ and $D_i=D$ for infinitely many $i$.
Also, we may assume $\bar x$ is the limit of the sequence $\{x_i\}_{i\ge 1}$ in $\bar X$, $\bar y$ is the limit of the sequence $\{y_i\}_{i\ge 1}$
in $\bar X$. Those two points are different contradicting
$d(x_n,y_n) < m$ for each $n\ge 1$. Indeed, give $\bar X$ a metric $\rho$ and define $f:X\to R$ as the distance to $\bar x$. It is extendible over $\bar X$, so it is slowly oscillating. Therefore $|f(x_n)-f(y_n)|\to 0$
but $f(x_n)\to 0$ and $f(y_n)\to \rho(\bar x,\bar y)\ne 0$.
\end{proof}

\begin{Corollary}
If $X$ is a proper geodesic space, then its Freundenthal compactification is a coarse compactification of $X$.
\end{Corollary}
\begin{proof}
Given a bounded subset $K$ of $X$ and given $m > 0$, put $L:=B(K,m)$
and notice $B(x,m)$ is a subset of $X\setminus K$ if $x\notin L$.
Therefore $B(x,m)$ is a subset of a component of $X\setminus K$ if $x\notin L$.
\end{proof}

\section{Glacial oscillations}
In this section we define a concept in the spirit of slowly oscillating functions and we use it to introduce coarse Freundenthal compactifications later on.
\begin{Definition}
A \textbf{glacial scale} on a metric space $X$ is a sequence of pairs $\{(K_i,n_i)\}_{i=1}^{\infty}$ of bounded subsets of $X$ and natural numbers
such that for each pair $(K,r)$ consisting of a bounded subset of $X$ and $r > 0$ there is $i$ such that $K\subset K_i$ and $n_i > r$.

Given a glacial scale $\mathcal{S}=\{(K_i,n_i)\}_{i=1}^{\infty}$, a chain of points $x_1,\ldots, x_n$ in $X$ is called an $\mathcal{S}$-\textbf{chain} if for each $i\leq n-1$ there is $m\ge 1$ such that $x_i, x_{i+1}\notin K_m$ and $d(x_i,x_{i+1}) \leq n_m$.
\end{Definition}

\begin{Definition}
A function $f:X\to R$ is \textbf{glacially oscillating} if for each $\epsilon > 0$ there is a glacial scale $\mathcal{S}$ with the property that
$|f(x_1)-f(x_n)| < \epsilon$ for each $\mathcal{S}$-chain $x_1,\ldots, x_n$.
\end{Definition}

\begin{Observation}
One can introduce the concept of a subset $A$ of $X$ to be $\mathcal{S}$-connected and reword the above definition as requiring that the diameter of $f(A)$ is less than $\epsilon$ for each $\mathcal{S}$-connected subset $A$ of $X$.
\end{Observation}

\begin{Proposition}
If $(X,d)$ is an ultrametric space, then every slowly oscillating function $f:X\to R$ is glacially oscillating.
\end{Proposition}
\begin{proof}
Recall that $(X,d)$ is an ultrametric space if every triangle in $X$ is isosceles and the lengths of two equal sides are at least the size of the third side.
Equivalently, $d(x,y)\leq \max(d(x,z),d(y,z))$ for all points $x,y,z\in X$.

If $f:X\to R$ is slowly oscillating and $\epsilon > 0$, then we can choose an increasing sequence $\{K_n\}_{n\ge 1}$ of non-empty bounded subsets of $X$ such that $B(K_n,n)\subset K_{n+1}$ for each $n\ge 1$ and
$|f(x)-f(y)| <\epsilon$ if $x,y\notin K_n$ and $d(x,y)\leq n$.
Let $\mathcal{S}:=\{(B(K_n,n),n)\}_{n\ge 1}$ and suppose $x_1,\ldots, x_n$ is an $\mathcal{S}$-chain. Pick $j\leq n-1$ such that $d(x_j,x_{j+1})$ is the maximum of
all $d(x_i,x_{i+1})$, $i\leq n-1$. Let $M$ be the smallest integer satisfying
$d(x_j,x_{j+1})\leq M$. Notice $x_j,x_{j+1}\notin B(K_M,M)$, the distance from $x_1$ to either $x_j$ or $x_{j+1}$ is at most $M$, the distance from $x_n$ to either $x_j$ or $x_{j+1}$ is at most $M$,
hence
$d(x_1,x_{n})\leq M$. That implies $x_1,x_n\notin K_M$, in particular
$|f(x_1)-f(x_n)| < \epsilon$. That proves $f$ is glacially oscillating.
\end{proof}

\begin{Proposition} \label{CloseGlacial}
Suppose $f:X\to X$ is close to $id_X$ and $h:X\to R$. If $h\circ f$ is glacially oscillating, then so is $h$.
\end{Proposition}
\begin{proof}
$f:X\to X$ being close to $id_X$ means there is $r > 0$ such that $d_X(f(x),x) < r$ for all $x\in X$.
Given $\epsilon > 0$ choose a glacial scale $\mathcal{S}=\{(K_i,n_i)\}_{i\ge 1}$
so that $|h\circ f(x_1)-h\circ f(x_2)| < \epsilon/3$ for any $x_1,x_2\in X$ that can be connected by an $\mathcal{S}$-chain in $X$. We may assume $n_1 > r$ by truncating $\mathcal{S}$. Let $\mathcal{S}'=\{B(K_i,r),n_i\}_{i\ge 1}$.
Now, given any $\mathcal{S}'$-chain $x_1,\ldots,x_n$
notice $x_1, f(x_1)$ and $x_n,f(x_n)$ are $\mathcal{S}$-chains.
Therefore $|h(x_1)-h(x_2)| < \epsilon$.
\end{proof}

\begin{Proposition}\label{GOAndCoarseEquivalences}
Suppose $f:(X,d_X)\to (Y,d_Y)$ is a coarse, large scale continuous function and $g:Y\to R$. \\
a. If $g$ is glacially oscillating, then so is $g\circ f$.\\
b. If $g\circ f$ is glacially oscillating and $f$ is a coarse equivalence, then $g$ is glacially oscillating.
\end{Proposition}
\begin{proof}
$f$ being coarse means $f^{-1}(K)$ is bounded for each bounded subset $K$ of $Y$. $f$ being large scale continuous means that for each $m\ge 0$ there is $M > 0$ such that $d_X(x,y) < m$ implies $d_Y(f(x),f(y)) < M$.

a. Given $\epsilon > 0$ choose a glacial scale $\mathcal{S}=\{(K_i,n_i)\}_{i\ge 1}$
so that $|g(y_1)-g(y_2)| < \epsilon$ for any $y_1,y_2\in Y$ that can be connected by an $\mathcal{S}$-chain in $Y$.
Put $C_i=f^{-1}(K_i)$ and let $m_i$ be the maximum of natural numbers
such that $d_X(x_1,x_2)\leq m_i$ implies $d_Y(f(x_1),f(x_2))\leq n_i$.
Notice $\mathcal{C}=\{(C_i,m_i)\}_{i\ge 1}$ is a glacial scale in $X$
and the image under $f$ of any $\mathcal{C}$-chain is an $\mathcal{S}$-chain. Therefore, if $x$ and $y$ can be connected by a $\mathcal{C}$-chain,
$f(x)$ and $f(y)$ can be connected by an $\mathcal{S}$-chain and $|g(f(x))-g(f(y))| < \epsilon$.

b. Choose $f':Y\to X$ that is coarse, large scale continuous such that
$f\circ f'$ is close to $id_Y$. By a), $g\circ f\circ f'$ is glacially oscillating
and by \ref{CloseGlacial}, so is $g$.
\end{proof}

\begin{Corollary}
If $(X,d)$ is a metric space of asymptotic dimension $0$, then every slowly oscillating function $f:X\to R$ is glacially oscillating.
\end{Corollary}
\begin{proof} As shown in
\cite{BDHM} there is an ultrametric space coarsely equivalent to $(X,d)$. Apply \ref{GOAndCoarseEquivalences}.
\end{proof}

\begin{Definition}
A subset $A$ of a metric space $X$ is \textbf{coarsely clopen} if $A$ and $X\setminus A$ are coarsely disjoint, i.e. the characteristic function $\chi_A$ of $A$ is slowly oscillating on $X$.
\end{Definition}

A basic property of coarsely clopen subsets of a metric space $X$ is the following:
\begin{Lemma}\label{AlgebraOfCCSets}
If $A$ and $C$ are coarsely clopen subsets of $X$, then so are $A\cap C$, $A\setminus C$, and $A\cup C$.
\end{Lemma}
\begin{proof}
Notice $\chi_{A\cap C}= \chi_A\cdot \chi_C$, $\chi_{A\setminus C}= \chi_A-  \chi_A\cdot \chi_C$, and $\chi_{A\cup C}= \chi_{A\setminus C}+ \chi_{C\setminus A}+\chi_A\cdot \chi_C$ are all slowly oscillating 
if both $\chi_A$ and $\chi_C$ are slowly oscillating.
\end{proof}

Here is a description of coarsely clopen subsets of proper metric spaces:

\begin{Proposition}\label{CClopenChar1}
If $(X,d)$ is a proper metric space and $A$ is a subset of its Higson compactification $h(X)$, then $A\cap X$ is coarsely clopen in $X$ if and only if $cl(A)\cap (h(X)\setminus X)$ and $cl(X\setminus A)\cap (h(X)\setminus X)$ are disjoint, where the closures are taken in $h(X)$.
\end{Proposition}
\begin{proof}
If $cl(A)\cap (h(X)\setminus X)$ and $cl(X\setminus A)\cap (h(X)\setminus X)$ are disjoint, then one cannot produce two disjoint sequences, $S_1:=\{x_n\}$ in $A$ and $S_2\{y_n\}$ in $X\setminus A$ such that $x_n\to\infty$ (that means each bounded subset $K$ of $X$ contains only finitely many members of the sequence)
and $d(x_n,y_n) < M$ for some $M > 0$ and all $n\ge 1$.
Indeed, in that case the closures of both sequences would have a common point in $ h(X)\setminus X$ as otherwise the characteristic function
of $cl(S_1)$ in $cl(S_1\cup S_2)$ extends over $h(X)$ to a continuous function $f:h(X)\to [0,1]$ and $f|X$ is slowly oscillating contradicting $f(x_n)=1$, $f(y_n)=0$ for all $n\ge 1$.

Conversely, if $C:=A\cap X$ is coarsely clopen, then the closure $cl_X(C)$ of 
$C$ intersects $cl_X(X\setminus C)$ along a bounded subset $K$,
so we may find a bounded open subset $U$ of $X$ such that
the characteristic function $\chi_(A\setminus U)$ is continuous on
$X\setminus U$ and is slowly oscillating. It extends to a continuous function on $h(X)\setminus U$
which is the characteristic function of $h(X)\cap cl(C\setminus U)$ (the closure taken in $h(X)$) when restricted to $h(X)\setminus U$. That proves $cl(A)\cap (h(X)\setminus X)$ and $cl(X\setminus A)\cap (h(X)\setminus X)$ being disjoint.
\end{proof}

\begin{Proposition}\label{CClopenChar2}
Given a subset $A$ of a metric space $X$ the following conditions are equivalent:\\
1. The characteristic function $\chi_A$ of $A$ is glacially oscillating,\\
2. $A$ is coarsely clopen,\\
3. There is a glacial scale $\mathcal{S}$ with the property that any $\mathcal{S}$-chain starting at a point of $A$ is completely contained in $A$.
\end{Proposition}
\begin{proof}
1)$\implies$2) is clear as $\chi_A$ of $A$ is glacially oscillating implies
$\chi_A$ of $A$ is slowly oscillating which is equivalent to $A$ being coarsely clopen.

2)$\implies$3). For each $n\ge 1$ choose a bounded subset $K_n$ containing $B(x_0,n)$ such that given two points $x,y\in X\setminus K_n$ at distance less than $n$, $x\in A$ implies $y\in A$.
Put $\mathcal{S}=\{(K_n,n)\}_{n\ge 1}$.

3)$\implies$1). Given $\epsilon > 0$ notice that for any $\mathcal{S}$-chain
$x_1,\ldots,x_n$ in $X$ one has $\chi_A(x_1)=\chi_A(x_n)$.
\end{proof}

\begin{Corollary}\label{OneGlacialScaleForMany}
Given finitely many coarsely clopen subsets $A_i$ of a metric space $X$, there is a glacial scale $\mathcal{S}$ with the property that any $\mathcal{S}$-chain starting at a point of some $A_j$ is completely contained in $A_j$.
\end{Corollary}
\begin{proof}
For each $j\leq n$ pick a glacial scale $\mathcal{S}_j=\{(K^j_i,k^j_i\}_{i\ge 1}$ with the property that any $\mathcal{S}_j$-chain starting at a point of $A_j$ is completely contained in $A_j$. Define $K_i=\bigcup\limits_{j=1}^n K^j_i$
and $k_i=\min(k^1_i,\ldots,k^n_i)$.
Notice that for $\mathcal{S}=\{(K_i,k_i)\}_{i\ge 1}$ any $\mathcal{S}$-chain starting at a point of some $A_j$ is completely contained in $A_j$.
\end{proof}

\begin{Corollary}
If $(X,d)$ is a metric space, then any slowly oscillating function $f:X\to R$
whose image is finite
is glacially oscillating.
\end{Corollary}
\begin{proof}
Notice point inverses of $f$ are coarsely clopen. Therefore $f$ is a linear combination of glacially oscillating functions and is itself glacially oscillating.
\end{proof}

\begin{Proposition}\label{CharOfGOForGeodesic}
If $X$ is a geodesic space and $f:X\to R$, then the following conditions are equivalent:\\
1. $f$ is glacially oscillating,\\
2. For each $\epsilon > 0$ there is a bounded subset $K$ of $X$ such that for every component $C$ of $X\setminus K$ the diameter of $f(C)$ is at most $\epsilon$.
\end{Proposition}
\begin{proof}
1)$\implies$2). Given $\epsilon > 0$ pick a glacial scale
$\mathcal{S}=\{(K_i,n_i)\}_{i\ge 1}$ with the property that 
$|f(x_1)-f(x_n)| < \epsilon/2$ for every $\mathcal{S}$-chain $x_1,\ldots,x_n$.
Put $K=B(K_1,n_1)$ and notice that every two points $x,y$ in a component $C$ of $X\setminus K$ can be connected by an $n_1$-chain in $X\setminus K_1$. That chain is also an $\mathcal{S}$-chain, so $|f(x)-f(y)| < \epsilon/2$
and $\diam(C) \leq \epsilon$.

2)$\implies$1). For each $n\ge 1$ pick a bounded set $K_n$ containing $B(x_0,n)\cup B(K_{n-1},n)$ such that the diameter of $f(C)$ is at most $1/n$ for each component $C$ of $X\setminus K_n$.
 Given $\epsilon > 0$ choose $k$ such that $1/k < \epsilon$ and consider
$\mathcal{S}=\{(K_{i+k},i+k)\}_{i\ge 1}$. Notice that any $\mathcal{S}$-component $C$
is contained in a component of $X\setminus K_k$, so $\diam(f(C)) < \epsilon$.
\end{proof}

\begin{Lemma}\label{ConstructingCCSets}
Suppose $\{K_i\}_{i\ge 1}$ is an increasing sequence of bounded subsets of a metric space $X$, $n_i$ is a strictly increasing sequence of natural numbers, and $A_i$ is an $n_i$-connected subset of $X\setminus K_i$ for each $i\ge 1$. If $A_i\cap (X\setminus K_{i+1})\subset A_{i+1}$ for each $i\ge 1$,
then $A:=\bigcup\limits_{i=1}^\infty A_i$ is a coarsely open subset of $X$.
\end{Lemma}
\begin{proof}
Suppose $x_n\in A$, $x_n\to\infty$, $y_n\notin A$ and $d(x_n,y_n) < M$ for each $n\ge 1$. By switching to subsequences of $x_n$ and $y_n$ we may assume there is a strictly increasing sequence $m(n)$ such that $x_n,y_n\notin K_{m(n)}$ for each $n\ge 1$ (otherwise infinitely many elements of $x_n$ would belong to the same bounded subset of $X$). Also, we may assume $m(1) > M$. There is $p\ge 1$ so that $x_1\in A_p$. 
If $p < m(1)$, then $x_1\in A_{m(1)}$, so we may assume $p\ge m(1)$.
Now, $y\in A_p$, a contradiction.
\end{proof}

\begin{Proposition}\label{MainCharOfGO}
Given a bounded function $f:X\to R$ on a metric space $X$, the following conditions are equivalent:\\
1. $f$ is glacially oscillating.\\
2. For every $\epsilon, M > 0$ there exist a bounded subset $K$ of $X$  such that $\diam(f(U)) < \epsilon$ for every $M$-component $U$ of $X\setminus K$.\\
3. For every $\epsilon > 0$ there exist a bounded subset $K$ of $X$ and finitely many coarsely clopen subsets $U_i$, $i\leq p$, covering $X\setminus K$ such that $\diam(f(U_i)) < \epsilon$ for each $i\leq p$.
\end{Proposition}
\begin{proof}
1)$\implies$2). Choose a glacial scale $\mathcal{S}=\{(K_i,n_i)\}_{i=1}^{\infty}$ for $f$ and $\epsilon/2$. There is $j\ge 1$ such that $n_j > M$.
If $U$ is an $M$-component of $X\setminus K_j$ then any two points of $U$ can be connected by an $\mathcal{S}$-chain. Therefore $\diam(f(U)) < \epsilon$.\\
2)$\implies$3). Choose $x_0\in X$. By induction create a sequence $n_i$ of natural numbers such that  $\diam(f(U)) < \epsilon/2^{i+1}$ for every $2i$-component $U$ of $X\setminus B(x_0,i)$ and $n_{i+1} > n_i+2i$ for each $i\ge 1$.

Cover $f(X)$ by $p$ intervals $I_j$ of size $\epsilon/2$. Given $j\leq p$ and $i\ge 1$ take the union $U_j^i$ of all $2i$-components $U$ of $X\setminus B(x_0,n_{i})$
such that $f(U)$ intersects $B(I_j,\epsilon/2-\epsilon/2^{j-1})$, where $I_j$ is the $j$-th interval. Put $U_j=\bigcup\limits_{i=1}^\infty U_j^i$. By \ref{ConstructingCCSets} each $U_j$ is a coarsely clopen subset of $X$. Notice that $\diam(f(U_j)) < \epsilon$.\\
3)$\implies$1). Given $\epsilon > 0$ choose a bounded subset $K$ of $X$ and finitely many coarsely clopen subsets $U_i$, $i\leq n$, covering $X\setminus K$ such that $\diam(f(U_i)) < \epsilon$ for each $i\leq n$.
We may assume $U_i$, $i\leq n$, are mutually disjoint by applying \ref{AlgebraOfCCSets}.
By \ref{OneGlacialScaleForMany}
 there is a glacial scale $\mathcal{S}$ with the property that any $\mathcal{S}$-chain starting at a point of some $U_j$ is completely contained in $U_j$. We may assume the first bounded set of $\mathcal{S}$ contains $K$.
Therefore, for any two points $x,y\in X$ joinable by an $\mathcal{S}$-chain
one has $|f(x)-f(y)| < \epsilon$ as they are contained in one of sets $U_i$, $i\leq n$. Thus $f$ is glacially oscillating.
\end{proof}

\section{Coarse Freundenthal compactification}
In this section we introduce the coarse Freundenthal compactification of proper metric spaces in a way similar to the Higson compactification. That approach should be of use to researchers in geometric group theory. Later on we will present a different approach that is more suitable for researchers in coarse topology.
\begin{Definition}
The \textbf{coarse Freundenthal compactification} of a proper metric space $X$ is the maximal compactification $CF(X)$ of $X$ with the property that any continuous and glacially oscillating function $f:X\to [0,1]$ extends over $CF(X)$ to a continuous function.
\end{Definition}

\begin{Proposition}
\label{InducedMapBetweenCoarseFreundenthal}
If $f:X\to Y$ is a coarse, large scale continuous function between proper metric spaces, then it extends uniquely to a continuous function of coarse Freundenthal compactifications.
\end{Proposition}
\begin{proof}
Given a continuous and glacially oscillating function $g:Y\to [0,1]$,
$g\circ f$ is also continuous and glacially oscillating by \ref{GOAndCoarseEquivalences}. Thus $f$ extends to a continuous function of coarse Freundenthal compactifications and that extension is unique as $X$ is dense in $CF(X)$.
\end{proof}

\begin{Proposition}\label{Corona0ImpliesGO}
If $\psi(X)$ is a coarse compactification of a proper metric space $X$ such that the corona $\psi(X)\setminus X$ is of dimension $0$, then for every continuous $f:\psi(X)\to R$ its restriction $f|X$ to $X$ is glacially oscillating.
\end{Proposition}
\begin{proof}
Given $\epsilon > 0$ consider mutually disjoint clopen subsets
$C_i$, $i\leq n$, covering $\psi(X)\setminus X$ such that
$\diam(f(C_i)) < \epsilon/2$ for each $i\leq n$. Extend each $C_i$ to an open subset $U_i$ of $\psi(X)$ such that $\diam(f(U_i)) < \epsilon$. Notice each $U_i$ is coarsely clopen, so by 
\ref{OneGlacialScaleForMany} there is a glacial scale $\mathcal{S}$ with the property
that any $\mathcal{S}$-chain starting at some $U_j$ is completely contained in that particular $U_j$. That proves $f|X$ is glacially oscillating
as we may assume each bounded subset used by $S$ contains
$\bigcup\limits_{i=1}^n (X\setminus U_i)$ which is a compact subset of $X$.
\end{proof}

\begin{Proposition}\label{GOImpliesCorona0}
Suppose $\psi(X)$ is a compactification of a proper metric space $X$. If for every continuous $f:\psi(X)\to R$ its restriction $f|X$ to $X$ is glacially oscillating, then the corona $\psi(X)\setminus X$ is of dimension $0$.
\end{Proposition}
\begin{proof}
Given two disjoint closed subsets $A_1$ and $A_2$ of $\psi(X)\setminus X$
choose a continuous function $f:\psi(X)\to [0,1]$ such that $f(A_1)\subset \{0\}$ and $f(A_2)\subset \{1\}$. Put $\epsilon = 1/4$
and choose a glacial scale $\mathcal{S}$ with the property that any $\mathcal{S}$-chain is mapped by $f$ to a subset of diameter less than $1/4$.
Put $U_1=f^{-1}[0,1/4]$ and $U_2=f^{-1}(3/4,1]$. Let
$V_i$, $i=1,2$, be the set of all points of $X$ than can be connected to $U_i$ by an $\mathcal{S}$-chain. Since we may assume each bounded set in the description of $S$ is compact, each $V_i$ is open in $X$.
Notice each $V_i$ is coarsely clopen (see \ref{CClopenChar2}) and $V_1\cap V_2=\emptyset$.
Therefore $C_i:=cl(V_i)\cap (\psi(X)\setminus X)$ is clopen in
$ \psi(X)\setminus X$ (see \ref{CClopenChar1}), it contains $A_i$, and $C_1\cap C_2=\emptyset$.
That proves $ \psi(X)\setminus X$ is of dimension $0$.
\end{proof}

\begin{Corollary}\label{CFIsCoronaOfDimZero}
The coarse Freundenthal compactification of a proper metric space $X$ is the maximal coarse compactification whose corona is of dimension $0$.
\end{Corollary}
\begin{proof}
By \ref{GOImpliesCorona0} the corona of the Freundenthal compactification is of dimension $0$. By \ref{Corona0ImpliesGO} any coarse compactification whose corona is of dimension $0$ is dominated by the Freundenthal compactification.
\end{proof}

\begin{Proposition}\label{CFAreHomeo}
If two proper metric spaces $(X,d_X)$ and $(Y,d_Y)$ are coarsely equivalent, then their Freundenthal coronas $CF(X)\setminus X$ and $CF(Y)\setminus Y$ are homeomorphic.
\end{Proposition}
\begin{proof}
Choose coarse equivalences $f:X\to Y$ and $g:Y\to X$ such that $g\circ f$ is close to $id_X$ and $f\circ g$ is close to $id_Y$.\\
\textbf{Case 1:} Both $X$ and $Y$ are discrete as topological spaces.
Using \ref{InducedMapBetweenCoarseFreundenthal} one can see that $f, g$ induce continuous function $CF(f):CF(X)\to CF(Y)$, $CF(g):CF(Y)\to CF(X)$.
It is known that the induced functions $h(f):h(X)\to h(Y)$,
$h(g):h(Y)\to h(X)$ have the property that $h(g)\circ h(f)$ is the identity
on $h(X)\setminus X$ and $h(f)\circ h(g)$ is the identity
on $h(Y)\setminus Y$. Since coarse Freundenthal compactifications are dominated by Higson compactifications, $CF(g)\circ CF(f)$ is the identity
on $CF(X)\setminus X$ and $CF(f)\circ CF(g)$ is the identity
on $CF(Y)\setminus Y$. \\
\textbf{Case 2:} $X$ is a discrete subset of $Y$ and $i:X\to Y$ is a coarse equivalence. Given a glacially oscillating function $k:X\to [0,1]$, it extends to a slowly oscillating and continuous function $k':Y\to [0,1]$ by \cite{DW} (see 
\cite{DM} for an earlier version of that result not involving continuity).
$k'$ being close to a glacially oscillating function (namely $k\circ r$, where $r:Y\to X$ is a coarse inverse of $i$) is itself glacially oscillating by \ref{CloseGlacial}. That means the closure $cl(X)$ of $X$ in $CF(Y)$ equals $CF(X)$ and $CF(X)\setminus X=CF(Y)\setminus Y$.\\
\textbf{General Case:} Choose discrete subsets $X'$ of $X$ and $Y'$ of $Y$ such that inclusions $X'\to X$ and $Y'\to Y$ are coarse equivalences. Apply Case 2 and then Case 1.
\end{proof}

\begin{Corollary}\label{FEqualsCF}
If $(X,d)$ is a proper geodesic space, then its Freundenthal compactification is the coarse Freundenthal compactification.
\end{Corollary}
\begin{proof}
Apply \ref{CharOfGOForGeodesic}.
\end{proof}

\section{Ends of groups}
In this section we show that the number of elements of the coarse Freundenthal corona of countable groups generalizes the number of ends of finitely generated groups.

Given a countable group $G$ one considers all proper metrics $d$ on $G$ that are left-invariant (that means $d(g\cdot h_1,g\cdot h_1)=d(h_1,h_2)$
for all $g,h_1,h_2\in G$). It turns out $id:(G,d_1)\to (G,d_2)$ is always a coarse equivalence for any such metrics $d_1$ and $d_2$. Be aware that considering right-invariant metrics $d_2$ (while keeping $d_1$ left-invariant) may lead to $id:G\to G$ not being a coarse equivalence (see \cite{BDM}).

If $G$ is finitely generated, then any word metric will do.

\begin{Example}\label{LeftInvariantMetric}
One way to introduce a proper left-invariant metric $d$ on a countable group $G$ that is not finitely generated is as follows: \\
1. Choose generators $g_i$, $i\ge 1$, of $G$ such that $g_{n+1}$ does not belong to the subgroup of $G$ generated by $g_1,\ldots,g_n$.\\
2. Choose a strictly increasing sequence of positive integers $n_i$,\\
3. Assign each $g_i^a$, where $a=\pm 1$, the norm $n_i$ and put $|1_G|=0$,\\
4. Assign to each $g\in G$ the norm minimizing the sum of norms
of $g_i^a$ appearing in all possible expressions of $g$ as the product of generators,\\
5. Define $d(g,h)$ as the norm of $g^{-1}\cdot h$.
\end{Example}

\begin{Observation}\label{AlmostInvariantObservation}
A subset $A$ of a countable group $G$ is coarsely clopen if and only if it is almost invariant.
\end{Observation}
\begin{proof}
Recall that $A$ is \textbf{almost invariant} if for every $g\in G$ the symmetric difference $A\Delta (A\cdot g)$ is finite.

Put a proper left-invariant metric $d$ on $G$.
Suppose $A$ is coarsely clopen but not almost invariant.
Choose $g\in G$ such that $A\Delta (A\cdot g)$ is not finite. Either there is a sequence $x_n$ in $A$ diverging to infinity such that $x_n\cdot g\notin A$ for all $n\ge 1$ or there is a sequence $x_n$ in $A$ diverging to infinity such that $x_n\cdot g^{-1}\notin A$ for all $n\ge 1$. Put $y_n=x_n\cdot g$ in the first case and $y_n=x_n\cdot g^{-1}$ in the second case. Observe $d(x_n,y_n)=d(1_G,g)$ for all $n\ge 1$ contradicting $A$ being coarsely clopen.

Suppose $A$ is almost invariant but not coarsely clopen. There is a sequence of points $x_n$ in $A$ diverging to infinity and $M > 0$ such that for some sequence $y_n\in G\setminus A$, $d(x_n,y_n) < M$ for all $n\ge 1$.
Therefore $x_n^{-1}\cdot y_n\in B(1_G,M)$ for all $n\ge 1$. Since $B(1_G,M)$ is finite, we may assume, without loss of generality, that there is $g\in B(1_G,M)$ such that $ x_n^{-1}\cdot y_n=g$ for all $n\ge 1$.
Hence $x_n\cdot g=y_n$ for all $n\ge 1$ and $A\Delta (A\cdot g)$ is not finite, a contradiction.
\end{proof}

\begin{Definition}
The \textbf{number of ends} of a countable group $G$ is the cardinality of the Freundenthal corona $CF(G)\setminus G$, where $G$ is equipped with a proper left-invariant metric.
\end{Definition}

Our next result shows that the above definition does generalize the classic definition of the number of ends for finitely generated groups.
\begin{Proposition}
If $G$ is finitely generated, then its number of ends equals the cardinality of the Freundenthal corona of any Cayley graph of $G$.
\end{Proposition}
\begin{proof}
Equip $G$ with a word metric based on a symmetric set of generators $S$. Notice the inclusion $G\to \Sigma(G,S)$ from $G$ to the Cayley graph
is a coarse equivalence, hence $CF(G)\setminus G$ is homeomorphic
to the coarse Freundenthal corona of $ \Sigma(G,S)$
(see \ref{CFAreHomeo}) and that corona is equal to the ends of $G$ by 
\ref{FEqualsCF}.
\end{proof}

Here is another way to introduce the number of ends of countable groups:
\begin{Proposition}
Let $G$ be a countable group. \\
1. $G$ has $0$ ends if it is finite.\\
2. $G$ has one end if every almost invariant subset $A$ of $G$ is either finite or it complement is finite.\\
3. The number of ends of $G$ is the supremum of $n\ge 0$ such that there
are $n$ mutually disjoint non-finite almost invariant subsets of $G$.
\end{Proposition}
\begin{proof}
If $CF(G)\setminus G$ has at least $n$ points, then it contains at least
$n$ non-empty clopen sets $C_i$ that are mutually disjoint.
We can extend them to mutually disjoint open subsets $U_i$ of $CF(G)$
such that $C_i=cl(U_i)\cap (CF(G)\setminus G)$ for each $i\leq n$.
By \ref{CClopenChar1} each $U_i$ is coarsely clopen and by \ref{AlmostInvariantObservation} each $U_i$ is almost invariant.

Cnversely, the existence of $n$ mutually disjoint non-finite almost invariant subsets of $G$ implies that $CF(G)\setminus G$ contains at least $n$ points.
\end{proof}

\begin{Proposition}\label{LocallyFiniteGroupsCase}
If $G$ is an infinite locally finite group, then its number of ends is infinite.
\end{Proposition}
\begin{proof}
Recall that a group $H$ is \textbf{locally finite} if each finite subset of it is contained in a finite subgroup of $H$. $H$ is locally finite if and only if its asymptotic dimension is $0$. By \ref{CFIsCoronaOfDimZero} the Higson compactification of $G$ is the coarse Freundenthal compactification of $G$. It is well-known that the Higson corona of unbounded metric spaces is infinite.
\end{proof}

\begin{Proposition}\label{SequenceOfSubgroupsProp}
Suppose a group $G$ is the union of an increasing sequence of its non-locally finite subgroups $\{G_i\}_{i\ge 1}$. If $A$ is a coarsely clopen infinite subset of $G$, then there is $n\ge 1$ such that $A\cap G_n$ is infinite.
\end{Proposition}
\begin{proof}
Suppose $A\cap G_i$ is finite for each $i\ge 1$. Since $G_1$ is not of asymptotic dimension $0$, there is $m\ge 1$ such that $G_1$ contains arbitrarily long $m$-chains. Choose $k\ge 1$ such that
$B(A,m)\cap B(G\setminus A,m)\subset B(1_G,k)$, then find $m > k$ such that $B(1_G,k)\subset G_m$ and then there is $a\in A_{m+1}\setminus G_m$.
Pick an $m$-chain $C$ in $G_1$ that is longer than the number of elements in $A_{m+1}$. By translating (i.e. switching from $C$ to $g\cdot C$ for some $g\in G$) we may assume $C$ starts at $1_G$.
Notice $a\cdot C$ is completely outside of $G_m$, so $a\cdot C\subset A$.
Hence $a\cdot C\subset A_{m+1}$, a contradiction.
\end{proof}

\begin{Definition}
\textbf{NCC} is a \textbf{shortcut for non-trivial coarsely clopen subsets} $Y$ of a metric space $X$, i.e. those coarsely clopen subsets that are infinite and $X\setminus Y$ is infinite.
\end{Definition}

\begin{Lemma}\label{ActingOnThreeEnds}
Suppose $G$ contains three NCC sets that are disjoint. If $G$ is not locally finite, then it acts trivially on at most one of the three NCC sets.
\end{Lemma}
\begin{proof}
$G$ \textbf{acts trivially} on an NCC set $E$ means the symmetric difference $E\Delta (g\cdot E)$ is finite for each $g\in G$.

Suppose $G$ acts trivially on disjoint NCC sets $E_1$, $E_2$ and $E_3$
is an NCC sets disjoint from $E_1\cup E_2$.
Using \ref{SequenceOfSubgroupsProp} we may reduce the proof to $G$ being finitely generated. Equip $G$ with a left-invariant word metric $d$.
Find a bounded subset $K$ of $G$ containing $1_G$ such that if $i\ne j$ and $g\in E_i\setminus K$, $h\in E_j\setminus K$, then $d(g,h) > 2$. Same for complements.
Let $E_4:=G\setminus (E_1\cup E_2\cup E_3)$. Either $E_4$ is an NCC or it is finite.
Find $m\ge 1$ such that for any $x\in E_i$, $i\leq 3$, of norm at least $m$, $B(x,2\cdot diam(K)+2)$
is contained in $E_i$. If $E_4$ is unbounded, require the same property for $E_4$, otherwise require that $B(x,diam(K)+1)$ is disjoint with $E_4$.

In $E_3$ find an element $g_3$ of the norm bigger than $m$. Hence
$g_3\cdot K\subset E_3$.

Since $E_1\Delta (g_3\cdot E_1)$ is finite,
choose $g_1\in E_1$ of the norm larger than $m$ such that $g_3\cdot g_1\in E_1$. Given a $1$-chain $c$ joining $g_1$ to $g_0\in K$, it stays in $E_1$ until it hits $K$ for the first time. Truncate $c$ to include only those elements of $G$. Now, $g_3\cdot c$ is a $1$-chain starting in $E_1$ and ending in $E_3$. Therefore it hits $K$ at certain moment. That means existence of $x_1\in E_1$ such that $g_3\cdot x_1\in K$.
Similarly, we can find $x_2\in E_2$ such that $g_3\cdot x_2\in K$. That means $g_3^{-1}\cdot K$ intersects both $E_1$ and $E_2$, a contradiction as that set is contained exclusively in only one of $E_i$, $i\leq 4$, due to the norm of
$g_3^{-1}$ being larger than $m$.
\end{proof}

\begin{Theorem}
If $G$ is a countable group, then the number of ends of $G$ is either infinite or at most $2$.
\end{Theorem}
\begin{proof}
If $G$ is finite, then $Ends(G)$ is empty.
If $G$ is locally finite and infinite, then $Ends(G)$ is infinite by \ref{LocallyFiniteGroupsCase}.

Assume $G$ is infinite, not locally finite, its number of ends is finite, and it contains three NCC sets
that are disjoint. Since $G$ acts on its ends via the left multiplication, there is a subgroup $H$ of $G$ of finite index that acts on $Ends(G)$ trivially. By \ref{CFAreHomeo} (Case 2) $H$ acts trivially on $Ends(H)$ which is equal to $Ends(G)$. This contradicts Lemma \ref{ActingOnThreeEnds}.
\end{proof}

\begin{Theorem}
Suppose $G$ is a countable group. If $G$ is the union of an increasing sequence $\{G_n\}_{n\ge 1}$ of its subgroups that have finitely many ends (that have at most $m$ ends), then the number of ends of $G$ is at most $2$ (is at most $m$).
\end{Theorem}
\begin{proof}
Suppose $G$ has $m+1$ mutually disjoint NCC sets $E_i$, $i\leq m+1$.
By
\ref{SequenceOfSubgroupsProp} we can find an index $n$ such that
each $E_i\cap G_n$ is an NCC set in $G_n$, a contradiction.
\end{proof}

\begin{Theorem}
If $G$ is a countable group with $2$ ends, then it is finitely generated. Therefore it contains an infinite cyclic subgroup of finite index.
\end{Theorem}
\begin{proof}
Suppose $G$ is a countable group with $2$ ends that is not finitely generated. We will show that there exists a subgroup $H$ of $G$ of index at most $2$ and a strictly increasing sequence $H_n$ of subgroups of $H$ satisfying the following conditions:\\
1. $H_1$ is infinite cyclic,\\
2. $H_n$ is of finite index in $H_{n+1}$ for each $n\ge 1$,\\
3. $H$ is the union of all $H_n$, $n\ge 1$.

$G$ acts on its ends $Ends(G)$ and it has a subgroup $H$ acting on $Ends(G)=Ends(H)$ trivially. Express $H$ as the union of two disjoint NCC sets $E_1, E_2$ which are almost invariant in $H$.
Given a finite subset $F$ of $H$ we can find using \ref{SequenceOfSubgroupsProp} a finitely generated subgroup $H_F$ of $H$ such that both $E_1\cap H_F$ and $E_2\cap H_F$ are NCC sets in $H_F$.
By a theorem of Mike Mihalik (see \cite{MM}, Theorem 1.2.12) $H_F$ cannot have infinitely many ends as for such groups  $H_F\cdot E$ is dense in $Ends(H_F)$ for any end $E$. In particular, there is $g\in H_F$ such that $(g\cdot E_1\cap H_F)\cap E_2$ is an NCC set in $H_F$, a contradiction. Thus $H_F$ has exactly $2$ ends. By a theorem of J.Stallings, $H_F$ has an infinite cyclic subgroup of finite index. In particular, if we construct two subgroups
$H_F\subset H_{F'}$ that way, then $H_F$ is of finite index in $H_{F'}$. Using these fact it is easy to construct the required sequence $H_n$ of subgroups of $H$.

Suppose $H$ has an NCC set $C$. There is $m\ge 1$ such that
$E\cap H_m$ and $E^c\cap H_m$ are both infinite. Let $t$ be a generator of $H_1$. Since both $E\Delta(E\cdot t)$ and $E^c\Delta(E^c\cdot t)$ are finite,
there is $k > m$ such that both these sets are contained in $H_k$.
Given $x\in E\cap H_{k+1}\setminus H_k$ one has $x\cdot t\in E$ as $x\cdot t\notin H_k$. Consequently, $x\cdot t^n\in E$ for all integer $n$.
The set $A:=\{x\cdot t^n\}_{n\in \mathbb{Z}}$ is isometric to $H_1$
as $d(x\cdot t^i,x\cdot t^j)=d(t^i,t^j)$ for all $i,j\in \mathbb{Z}$.
Let $f:H_{k+1}\to H_1$ be a coarse inverse of the inclusion $H_1\to H_{k+1}$.
Since $f|A:A\to H_1$ is a coarse embedding, it must be a coarse equivalence.
Therefore the inclusion $A\to H_{k+1}$ is a coarse equivalence
and $A\cap E$ ought to be an NCC set in $A$ contradicting $A\subset E$.
\end{proof}

\begin{Corollary}
The group of rational numbers has $1$ end. More generally, any countable subgroup of reals has one end if it is not finitely generated.
\end{Corollary}

\section{Ends of coarse spaces}
In this section we generalize the concept of Freundenthal ends to arbitrary coarse spaces. See \cite{JDEnds} for other ways to introduce ends in coarse spaces.

We follow a description of coarse spaces (quite often our terminology is that of \textbf{large scale spaces}) as in \cite{DH}. It is equivalent to Roe's definition of those spaces in \cite{Roe lectures}.

Recall that a \textbf{star} $st(x,U)$ of $x\in X$ with respect to a family $\mathcal{U}$ of subsets of $X$ is defined as the union of $U\in \mathcal{U}$ containing $x$. If $A\subset X$, then $st(A,\mathcal{U})=\bigcup\limits_{x\in A}st(x,\mathcal{U})$. Given two families $\mathcal{U},\mathcal{V}$ of subsets of $X$,
$st(\mathcal{U},\mathcal{V})$ is defined as the family $st(A,\mathcal{V})$, $A\in \mathcal{U}$.

\begin{Definition}
A \textbf{large scale space} is a set $X$ equipped with a family $\mathbb{LSS}$ of covers (called \textbf{uniformly bounded} covers) satisfying the following two conditions:\\
1. 
$st(\mathcal{U},\mathcal{V})\in \mathbb{LSS}$ if $\mathcal{U},\mathcal{V}\in \mathbb{LSS}$.\\
2. If $\mathcal{U}\in \mathbb{LSS}$ and every element of $\mathcal{V}$ is contained in some element of $\mathcal{U}$, then $\mathcal{V}\in \mathbb{LSS}$.

Sets which are contained in an element of $\mathcal{U}\in \mathbb{LSS}$ are called \textbf{bounded}.
\end{Definition}

\begin{Definition}
A subset $A$ of a large scale space $X$ is \textbf{coarsely clopen} if for every uniformly bounded cover $\mathcal{U}$ of $X$ the set $st(A,\mathcal{U})\cap st(X\setminus A,\mathcal{U})$ is bounded.

A \textbf{non-trivial coarsely clopen} subset $A$ of a large scale space $X$ (an NCC-set for short) is one that is not bounded and $X\setminus A$ is not bounded.
\end{Definition}

\begin{Lemma}\label{StLemma}
$st(A_1\cap A_2,\mathcal{U})\cap st((A_1\cap A_2)^c,\mathcal{U})\subset
st(A_1,\mathcal{U})\cap st((A_1)^c,\mathcal{U})\cup st(A_2,\mathcal{U})\cap st((A_2)^c,\mathcal{U})$.
\end{Lemma}
\begin{proof}
Suppose $x\in st(A_1\cap A_2,\mathcal{U})\cap st((A_1\cap A_2)^c,\mathcal{U})$. There is $y\in A_1\cap A_2$ satisfying $x\in st(y,\mathcal{U})$
and there is $z\in A_1^c\cup A_2^c$ satisfying $x\in st(z,\mathcal{U})$.
Thus either $x\in st((A_1)^c,\mathcal{U})$ or $x\in st((A_2)^c,\mathcal{U})$
and we are done.
\end{proof}

\begin{Corollary}
\label{IntersectionOFCCs}
The intersection of two coarsely clopen subsets of $X$ is coarsely clopen.
\end{Corollary}
\begin{proof}
Apply \ref{StLemma}.
\end{proof}

\begin{Definition}
A topology on $X$ is \textbf{compatible} with the large scale structure on $X$ if there is a uniformly bounded cover of $X$ consisting of open subsets of $X$. 
\end{Definition}

\begin{Observation}
The simplest non-trivial topology compatible with a large scale structure is the discrete topology.
\end{Observation}

\begin{Definition}
A \textbf{topological coarse space} is a set equipped with large scale structure and with a compatible topology. Additionally, we assume that the coarse structure is \textbf{coarsely connected}, i.e. the union of two bounded subsets of $X$ is bounded.
\end{Definition}

\begin{Lemma}\label{ExtendingNCCLemma}
If $A$ is a coarsely clopen subset of $X$, then $st(A,\mathcal{U})$ is a coarsely clopen subset of $X$ for each uniformly bounded cover $\mathcal{U}$ of $X$.
\end{Lemma}
\begin{proof}
Notice $st(st(A,\mathcal{U}),\mathcal{V})\subset st(A,st(\mathcal{U},\mathcal{V}))$ for any two covers $\mathcal{U},\mathcal{V}$.
Therefore $st(st(A,\mathcal{U}),\mathcal{V})\cap st(st(A^c,\mathcal{U}),\mathcal{V})\subset st(A,st(\mathcal{U},\mathcal{V}))\cap st(A^c,st(\mathcal{U},\mathcal{V}))$. Since
$st(A,\mathcal{U})^c\subset A^c\subset st(A^c,\mathcal{U})$ the proof is completed.
\end{proof}

\begin{Lemma}\label{ShrinkingNCCLemma}
If $A$ is a coarsely clopen subset of $X$, then a subset $C$ of $A$ is coarsely clopen provided
 $A\subset st(C,\mathcal{V})$ for some uniformly bounded cover $\mathcal{V}$ of $X$.
\end{Lemma}
\begin{proof}
Observe $C':=(st(A^c,\mathcal{V}))^c\subset C$ is coarsely clopen by \ref{ExtendingNCCLemma} and $B:=C\setminus C'\subset A\cap st(A^c,\mathcal{V})$ is bounded as $st(A,V)\cap st(A^c,\mathcal{V})$
is bounded. Adding a bounded set $B$ to a coarsely clopen subset
preserves being coarsely clopen as can be easily seen.
\end{proof}

\begin{Definition}
An \textbf{end} of a topological large scale space $X$ is a family $E$ of unbounded, open and coarsely clopen subsets of $X$ that is maximal with respect to the property of being closed under intersections.
\end{Definition}

\begin{Definition}
Let $\mathcal{T}$ be the topology of a topological large scale space $X$, we \textbf{extend the topology} $\mathcal{T}$ over $X\cup Ends(X)$ as follows: $Y\subset X\cup Ends(X)$ is declared open if $Y\cap X$ is open in $X$ end for each end $E\in Y$ there is an open coarsely clopen set $U$ such that $U\in E$ and $U\subset Y$.
\end{Definition}

\begin{Proposition}\label{TopologyOnEndsIsIndependent}
The topology on $Ends(X)$ is independent of the topology on $X$ as long as the topology is compatible with the coarse structure.
\end{Proposition}
\begin{proof}
Suppose $U$ is an open and uniformly bounded cover of $X$.
If $E$ is an end of $X$ in the discrete topology, then there is a unique
end $E'$ of $X$ containing all sets $st(A,\mathcal{U})$, $A\in E$ (use \ref{ExtendingNCCLemma} and \ref{ShrinkingNCCLemma}).
Therefore $Ends(A)$ and $Ends(st(A,\mathcal{U}))$ are identical for any subset $A$ of $X$ and the proof is completed.
\end{proof}

Recall that a compact space is totally disconnected if its components are singletons. Equivalently, it has a basis of open closed subsets (see \cite{Engel}) which is our preferred point of view.

\begin{Proposition}\label{LSCompactness}
1. $X\cup Ends(X)$ is large scale compact.\\
2. $Ends(X)$ is compact Hausdorff and totally disconnected.\\
3. $X\cup Ends(X)$ is Hausdorff if $X$ is Hausdorff.
\end{Proposition}
\begin{proof}
$X\cup Ends(X)$ being large scale compact means that for any open cover
$\{U_s\}_{s\in S}$ of it there is a finite subset  $F$ of $S$ such that
$Ends(X)\subset \bigcup\limits_{s\in F}U_s$ and $X\setminus \bigcup\limits_{s\in F}U_s$ is a bounded subset of $X$ (see \cite{JDUnifying}).

\textbf{Claim 1:} Given a family $\{U_s\}$ of open coarsely clopen subsets of $X$ such that $Ends(X)\subset \bigcup\limits_{s\in S}Ends(U_s)$, there is a finite subset  $F$ of $S$ such that
$Ends(X)\subset \bigcup\limits_{s\in F}Ends(U_s)$ and $X\setminus \bigcup\limits_{s\in F}U_s$ is a bounded subset of $X$.\\
\textbf{Proof of Claim 1:} 
Consider a uniformly bounded and open cover $\mathcal{V}$ of $X$.
Let $V_s:=st(U_s,\mathcal{V})$, $C_s:=cl(U_s)$ for each $s\in S$. Those are coarsely clopen subsets of $X$ by \ref{ExtendingNCCLemma} and by \ref{ShrinkingNCCLemma}
as $C_s\subset V_s$ for each $s\in S$.

Consider the family $X\setminus \bigcup\limits_{s\in F}C_s$, $F$ a finite subset of $S$. It cannot be extended to an end of $X$ as such an end cannot belong to $\bigcup\limits_{s\in F}Ends(U_s)$, so there is $F$ such that
$B:=X\setminus \bigcup\limits_{s\in F}C_s$ is bounded.

To show $C:=X\setminus \bigcup\limits_{s\in F}U_s$ is bounded
define $B_s$ as $st(U_s,\mathcal{V})\cap st(X\setminus U_s,\mathcal{V})$.
We plan to show $C\subset B\cup\bigcup\limits_{s\in F}B_s$.
Suppose $x\in C\setminus B$. There is $t\in F$ so that $x\in C_t$,
hence $x\in C_t\setminus U_t$. Therefore $x\in st(U_t,\mathcal{V})\cap st(X\setminus U_t,\mathcal{V})=B_t$. Thus $C$ is bounded.

Now, $\bigcup\limits_{s\in F}Ends(U_s)$ must contain $Ends(X)$
as otherwise there is an end $E$ containing each $X\setminus C_s$, $s\in F$, hence also containing $X\setminus \bigcup\limits_{s\in F}C_s$, a contradiction.

\textbf{Claim 2:} $X$ is an open subset of $X\cup Ends(X)$.\\
\textbf{Proof of Claim 2:}
 Suppose $x\in X$ and $E$ is an end of $X$. Pick a uniformly bounded cover $\mathcal{U}$ of $X$ consisting of open subsets. Let $x\in V\in \mathcal{U}$. Notice $cl(V)$ is bounded (as it is contained in $st(V,\mathcal{U})$, so $W:=X\setminus cl(V)\in E$ and $(W\cup Ends(W))\cap V=\emptyset$.

\textbf{Claim 3:} Two different ends of $X$ have disjoint neighborhoods in $X\cup Ends(X)$.\\
\textbf{Proof of Claim 3:}
Suppose $E_1\ne E_2$ are two different ends of $X$. There is $U\in E_1\setminus E_2$, hence there is $V\in E_2$ such that $U\cap V$ is bounded in view of \ref{IntersectionOFCCs}. Choose a uniformly bounded cover $\mathcal{U}$ of $X$ consisting of open subsets. Let $\mathcal{W}:=st(\mathcal{U},\mathcal{U})$. By \ref{ShrinkingNCCLemma}
both $A_U:=U\setminus st(U\cap V,\mathcal{W})$ and $A_V:=V\setminus st(U\cap V,\mathcal{W})$ are coarsely clopen. Notice
$U':=st(A_U,\mathcal{U})\in E_1$, $V':=st(A_V,\mathcal{U})\in E_2$ are disjoint, hence $U'\cup Ends(U')$ is a neighborhood of $E_1$,
$V'\cup Ends(V')$ is a neighborhood of $E_2$ and they are disjoint.

1. Follows from Claim 1.

2. $Ends(X)$ being compact Hausdorff follows from Claims 1-3.
Suppose $V$ is a neighborhood of the end $E$ in  $X\cup Ends(X)$.
Choose an open coarsely clopen subset $U$ so that $U\in E$ and
$U\cup Ends(U)\subset V$. Choose a bounded and open cover $\mathcal{W}$ of $X$ such that $st(U,\mathcal{W})\cap st(X\setminus U,\mathcal{W})$ is bounded. By \ref{ExtendingNCCLemma} the set $st(X\setminus U,\mathcal{W})$
is open and coarsely clopen. Notice $Ends(st(X\setminus U,\mathcal{W}))\cap Ends(U)=\emptyset$ and their union is $Ends(X)$. Thus
$Ends(U)$ is clopen in $Ends(X)$ and $Ends(X)$ is totally disconnected.

3. Follows from Claim 3 and the proof of Claim 2.
\end{proof}

\begin{Proposition}\label{ExtendingOverEnds}
Any continuous, coarse and large scale continuous function $f:X\to Y$ of topological coarse spaces extends to a continuous map $\bar f:X\cup Ends(X)\to Y\cup Ends(Y)$. If $f, g:X\to Y$ are close, then
$\bar f|Ends(X)=\bar g|Ends(Y)$.
\end{Proposition}
\begin{proof}
Given an end $E$ of $Y$ the family $f^{-1}(E)$ consists of unbounded coarsely clopen subsets of $X$, so any end of $X$ containing that family is mapped by $\bar f$ to $E$. It is clear $\bar f$ is continuous.

Suppose $f, g:X\to Y$ are close. There is an open uniformly bounded cover $\mathcal{U}$ of $Y$ with the property that $f(x)\in st(g(x),\mathcal{U})$ for all $x\in X$. Suppose $\bar f(E)\ne \bar g(E)$ for some end $E$ of $X$.
As in \ref{ExtendingNCCLemma} there are $V\in \bar f(E)$
and $W\in \bar g(E)$ such that $st(V,\mathcal{U})\cap st(W,\mathcal{U})=\emptyset$. Therefore $f^{-1}(V)\cap g^{-1}(W)=\emptyset$
contradicting $f^{-1}(V),g^{-1}(W)\in E$.
\end{proof}

\begin{Corollary}
If two topological coarse spaces $X$ and $Y$ are coarsely equivalent, then 
$Ends(X)$ is homeomorphic to $Ends(Y)$.
\end{Corollary}
\begin{proof}
By \ref{TopologyOnEndsIsIndependent} we may assume both $X$ and $Y$ are equipped with the discrete topology. Apply \ref{ExtendingOverEnds}.
\end{proof}

\begin{Lemma}\label{FindingCCSetsForGO}
If $(X,d)$ is a metric space and $f:X\to R$ is glacially oscillating, then for each compact subset $C$ of an open set $U\subset R$ there is a coarsely clopen subset $A$ of $X$ such that $f^{-1}(C)\subset A\subset f^{-1}(U)$.
\end{Lemma}
\begin{proof}
Choose $\epsilon > 0$ such that $B(C,\epsilon)\subset U$. Then choose a glacial scale $\mathcal{S}$ such that $|f(x)-f(y)| < \epsilon$ if $x$ and $y$ can be connected by an $\mathcal{S}$-chain.
Define $A$ as all points in $X$ that can be connected to $f^{-1}(C)$ by an $\mathcal{S}$-chain. Clearly, $f^{-1}(C)\subset A\subset f^{-1}(U)$
and $A$ is coarsely clopen by \ref{CClopenChar2}.
\end{proof}

\begin{Theorem}
If $(X,d)$ is a proper metric space, then $id_X:X\to X$ extends to a homeomorphism from the coarse Freundenthal compactification $CF(X)$ to $X\cup Ends(X)$.
\end{Theorem}
\begin{proof}
It suffices to show that for any glacially oscillating function $f:X\to [0,1]$, for any end $E$ of $X$ the set $C:=\bigcap\limits_{A\in E}cl(f(A))$ consists of one point. Suppose there are two different points $a,b\in C$ and put $\epsilon=|a-b|/4$. 
By \ref{FindingCCSetsForGO} there is a coarsely clopen subset $A'$ of $X$
such that $f^{-1}[a-\epsilon,a+\epsilon]\subset A'\subset f^{-1}(a-3\epsilon,a+3\epsilon)$. The set $K:=f^{-1}[a-\epsilon,a+\epsilon]$ cannot be bounded as in such case removing it from elements of $E$ would contradict
$a\in C$. Similarly, the complement of $A'$ cannot be bounded.
Since $A'\notin E$, there is $A\in E$ such that $A\cap A'$ is bounded.
Hence $A\setminus A'\in E$ and $b\notin cl(f(A\setminus A'))$, a contradiction.
\end{proof}

\begin{Remark}
Now we can extend the definition of the space of ends of arbitrary group $G$ by giving it the following large scale structure: uniformly bounded covers are those that refine covers of the form $\{g\cdot F\}_{g\in G}$ for some finite subset $F$ of $G$. The same can be done for locally compact topological groups. Instead of $F$ being finite we consider neighborhoods of $1_G$ with compact closure.
\end{Remark}

\end{document}